\documentclass{amsart}
\usepackage{amsmath}
\newtheorem{theorem}{Theorem}

\theoremstyle{definition}

\newcommand{\bbF}{\mathbb F}
\newcommand{\bF}{\bar{F}}
\newcommand{\bV}{\bar{V}}
\newcommand{\fF}{{\mathfrak F}} 
\newcommand{\p}{{\mbox{$[p]$}}}
\newcommand{\scp}{{\mbox{$\scriptstyle [p]$}}}
\DeclareMathOperator{\Hom}{Hom}
\DeclareMathOperator{\cl}{Cl}
\DeclareMathOperator{\eval}{Eval}
\title{On $\fF$-hypercentral modules and character clusters}
\author{Donald W. Barnes}
\address{1 Little Wonga Rd.\\Cremorne NSW 2090\\Australia\\}
\email{D.Barnes@maths.usyd.edu.au}


\begin{document}

\begin{abstract} Let $\mathfrak F$ be a saturated formation of soluble Lie algebras over a field $F$ of characteristic $p > 0$  and let ${\mathbb F}_p$ denote the field of $p$ elements.  Let $(L,[p])$ be a restricted Lie algebra over $F$ with $z^{\mbox{$\scriptstyle [p]$}}=0$ for all $z$ in the centre of $L$. Let $S \in \mathfrak F$, $S\ne 0$ be a subnormal subalgebra of $L$.  Let $V, W$ be  $L$-modules. Suppose that the character cluster of $W$ is contained in the set of ${\mathbb F}_p$-linear combinations of the characters in the character cluster  of $V$.  Suppose that $V$, regarded as $S$-module, is $\mathfrak F$-hypercentral. Then $W$, regarded as $S$-module, is also $\mathfrak F$-hypercentral.
\end{abstract}

\maketitle

Let $F$ be a field of characteristic $p>0$ and let $\fF$ be a saturated formation of soluble Lie algebras over $F$.  Let $(L,\p)$ be a restricted Lie algebra over $F$ with $z^\scp=0$ for all $z$ in the centre of $L$.  Suppose that $S \in \fF$, $S \ne 0$ is a subnormal subalgebra of $L$.  If  the $L$-module $V$ is a \p-module, then by \cite[Theorem 6.4]{extras}, $V$ is $S\fF$-hypercentral, that is, $\fF$-hypercentral as $S$-module.  As an $L$-module is a \p-module if and only if it has character $0$, this establishes a link between characters and $\fF$-centrality.  

An $S\fF$-hypercentral $L$-module need not have character $0$.  Not every $L$-module has a character.  To cope with this situation, the concept of the character cluster $\cl(V)$ of $V$ was defined in \cite[Section 2]{cluster} to be the set of characters of the composition factors of the module $\bV$ obtained by extending  the field to its algebraic closure $\bF$.  Since for any saturated formation $\fF$, the trivial $1$-dimensional module is $S\fF$-central and has character $0$, the following theorem is a generalisation of the  theorem cited above.   We denote the field of $p$ elements by $\bbF_p$ and the set of $\bbF_p$-linear combinations of the elements of $\cl(V)$ by $\bbF_p\cl(V)$.

\begin{theorem} \label{main} Let $\fF$ be a saturated formation of soluble Lie algebras over a field $F$ of characteristic $p > 0$. Let $(L,\p)$ be a restricted Lie algebra over $F$ with $z^\scp=0$ for all $z$ in the centre of $L$. Suppose that $S \in \fF$, $S \ne 0$ is a subnormal subalgebra of $L$. Let $V, W$ be $L$-modules with $\cl(W) \subseteq \bbF_p\cl(V)$ and suppose that $V$ is $S\fF$-hypercentral.  Then $W$ is $S\fF$-hypercentral.  
\end{theorem}

\begin{proof} Suppose that $\cl(V)$ has $k$ elements.  Put $X = \oplus_{r=1}^{k(p-1)} V^{\otimes r}$.  From \cite[Theorem 5.2.7(3)]{SF}, it follows that $\cl(V^{\otimes r}) = \{c_1+ \dots +c_r \mid c_i \in \cl(V) \}$ and so, that $\cl(X) = \bbF_p\cl(V)$.  By \cite[Theorem 2.1]{HyperC}, $X$ is $S\fF$-hypercentral.  By replacing $V$ with $X$, we may suppose that $\cl(W) \subseteq \cl(V)$.  We need only consider the case where $W$ is irreducible. 

 Suppose that the character $c \in \cl(W)$.  Then also $c \in \cl(V)$.  Denote the action of $x \in L$ on $\Hom(V,W)$ by $\rho(x)$.  By \cite[Theorem 5.2.7(1)]{SF}, $0 =c-c \in \cl(\Hom(V,W))$.  Thus the set $\{(\rho(x)^p - \rho(x^\scp)) f = 0 \mid x \in L\}$ of linear equations over $F$ has a non-zero solution for $f$ in $\bF \otimes \Hom(V,W)$, so it has a non-zero solution in $\Hom(V,W)$.  Thus $\Hom(V,W)$ has a non-zero submodule $H$ with character $0$.  The evaluation map $\eval: V \otimes H \to W$ given by $\eval(v,f) = f(v)$ has non-zero image.  As $\eval$ is an $L$-module homomorphism and $W$ is irreducible, we have $\eval(V \otimes H) =W$.  By \cite[Theorem 6.4]{extras}, $H$ is an $S\fF$-hypercentral $L$-module.  As $V$ is $S\fF$-hypercentral, by \cite[Theorem 2.1]{HyperC},  $V \otimes H$ is $S\fF$-hypercentral.  It follows that $W$ is $S\fF$-hypercentral.
\end{proof}

Suppose that we have $S \in \fF$, $S \ne 0$ and $S$-modules $V,W$ with $V$ $\fF$-hypercentral.  We can take any $p$-envelope $(L,\p)$ of $S$ and, if necessary, adjust the $p$-operation to have $z^\scp=0$ for all $z$ in the centre of $L$.  We then have that $S$ is an ideal of $L$.  Using the construction in the proof of  \cite[Proposition 2.5.6]{SF}, we can extend the actions of $S$ on $V,W$ making them into $L$-modules.  (This extension of the action is set out in detail in \cite[Section 4]{induced}.)  Suppose that for some choice of $(L,\p)$ and some choices of the extensions of the actions, we have that $\cl(W) \subseteq \bbF_p\cl(V)$.  Then it follows by Theorem \ref{main}, that $W$ also is $\fF$-hypercentral.

\bibliographystyle{amsplain}

\end{document}